\documentclass[a4paper,fleqn]{article} 

\usepackage[authoryear,longnamesfirst]{natbib} 

\usepackage{mathrsfs}

\usepackage{xspace} \usepackage{etoolbox} \def\tsc#1{\csdef{#1}{\textsc{\lowercase{#1}}\xspace}} 

\tsc{WGM} \tsc{QE} \tsc{EP} \tsc{PMS} \tsc{BEC} \tsc{DE} 

\usepackage{bm,color,graphicx}

\usepackage{amsmath,amssymb,amsthm,amsfonts}

\newcommand{\rar}{\rightarrow}  
\newcommand{\vett}{\mathbf}  
 \newcommand{\E}{\mathbb E\,}
\newcommand{\N}{\mathbb N} \renewcommand{\P}{\mathbb P} 
 \newcommand{\R}{\mathbb R} 
  
\newcommand{\Chi}{{\hbox{\rm 1 \kern-6.4truept l}}}

\theoremstyle{plain} \newtheorem{theorem}{Theorem}[section] 
\newtheorem{hyp}[theorem]{Hypothesis} 
 
\newtheorem{definition}[theorem]{Definition} 
\newtheorem{lemma}[theorem]{Lemma} 
\newtheorem{prop}[theorem]{Proposition} 
\newtheorem{example}[theorem]{Example} 
\newtheorem{remark}[theorem]{Remark} 

\title{Asymptotic Behaviour for Isotropic Pearson Random Walks} 

\author{ D. A. Bignamini\thanks{DiSAT, Università dell'Insubria, Via Valleggio 11, 22100 Como, Italy.}\thanks{Corresponding author. E-mail: \texttt{da.bignamini@uninsubria.it}} \and E. G. Casini\footnotemark[1] \and A. Martinelli\footnotemark[1] }

\date{} 

\begin{document} 

\maketitle 

\begin{abstract} In the Pearson random walk the direction of the $i$--th step is a random variable uniformly distributed on the unit sphere of $\R^d$, and its length is a non-negative random variable. We prove that, when $d\to+\infty$, the Euclidean norm of the walk converges almost surely to the Euclidean norm of the vector of step lengths, whose joint law is arbitrary. Further, we investigate the same convergence result in a more general framework, where both the number and the lengths of the steps depend on the dimension. The results in the present paper can be applied to approximate the distribution of some classical Pearson random walks, for example, in the Dirichlet case.
\end{abstract} 

\bigskip 


\section{Introduction}
In the original formulation of random walks by \cite{Pearson:1905p18079}, a particle moving in the plane chooses uniformly the direction of its motion and performs $n$ independent steps of equal length.
More generally, in this paper we study the random walk
\[
\vett S_{d,n} := \sum_{i = 1}^{n} L_i\vett U_i^{(d)},
\]
where $\vett U_i^{(d)}$ is a random vector having uniform distribution on the Euclidean sphere in $\R^d$ that represents the direction of the $i$-th step and $\vett L := ( L_1,L_2,\dots,L_n)$ is a non-negative random vector (the variables can be dependent), where $L_i$ represents the length of the $i$-th step. These kinds of processes (and their variants) have been widely studied, and determining their distributions and probabilistic properties is very difficult, see, for instance, \cite{casini:2015p254}, \cite{LeCaer:2010p16914} and \cite{Orsingher:2007p17062}. We are interested in the case where the dimension $d$ is high, this case has many applications in physics, see \cite{Weiss:2002hf} for a review or \cite{Hughes:1995hd}. One motivating example is the following.

\begin{example}[\cite{le_caer_two-step_2015}]
    	In this example, we consider the two-step random walk with symmetric Dirichlet random lengths studied in \cite{le_caer_two-step_2015}.\\
        We denote by $|\cdot|_2$ the Euclidean norm. In this case $\vett L = \left( L, 1 - L \right)$, where $(L,1-L)$ has a symmetric Dirichlet distribution with parameter $q>0$, namely $L\sim{\rm Beta}(q,q)$. Under these assumptions, the support of the random variable $D := \left | \vett L \right |_2 = \sqrt{1 - 2 L (1 - L)}$ is the interval $\left[ 1/\sqrt{2}, 1 \right]$, and its distribution function is given by
   \[
	F_D\left( r \right) = \P\left( D\le r \right)
			= 2 F_L\left( \frac{1 + \sqrt{2 r^2 - 1}}{2} \right) - 1,\qquad r\in\left[1/\sqrt{2},1\right].
	\]
Furthermore, the distance of $\vett S_{d,2}$ from the origin is
\[
\left| \vett S_{d,2} \right|_2 = \sqrt{1 - Y Z^{(d)}}
\]
where $Y := 4 L \left( 1 - L \right)$ and $Z^{(d)} = \left( 1 - \left< \vett U_1^{(d)} , \vett U_2^{(d)} \right>\right)/2  \sim {\rm Beta}\left( \frac{d - 1}{2}, \frac{d -1}{2} \right)$ are independent random variables. It is easy to prove that $Z^{(d)}$ converges in distribution to $1/2$ as $d \rar \infty$, and therefore $\left| \vett S_{d, 2} \right|_2$ converges in distribution to $\left|\vett L\right|_2$.

\end{example}
We emphasize that, if one wishes to prove the same result for a random walk with a number of steps greater than two, the calculations become significantly more complex. In general, it is very difficult to determine the distribution of the Euclidean norm of $\vett S_{d,n}$ when the number of steps exceeds 2, see for instance \cite{letac_piccioni:2014}.

The asymptotic behaviour of $|\vett S_{d,n}|_2$ as $d\to+\infty$ is known in some particular cases. For steps of unit length, namely $\vett L=(1,\ldots,1)$, and fixed $n$, asymptotic formulas for the density and the moments of $|\vett S_{d,n}|_2$ are given in \cite{BorweinStraubVignat2016}. If the step lengths are independent and identically distributed, with law independent of $d$, the convergence in distribution of $|\vett S_{d,n}|_2^2$ to $|\vett L|_2^2$ for fixed $n$ is observed in \cite[Remark 3.4]{KabluchkoMarynych2024}, while laws of large numbers for $|\vett S_{d,n}|_2^2/n$ as both $n$ and $d$ go to $+\infty$ and central limit theorems are studied in \cite{Grundmann2014, KabluchkoMarynych2024, RoslerVoit2011}. Related results on the almost orthogonality of independent random vectors in high dimension can be found in \cite{CaiFanJiang2013} and \cite{HallMarronNeeman2005}.

In the results on random walks recalled above, the step lengths are deterministic or independent and identically distributed. In the present paper, the joint law of the step lengths is arbitrary, so that they may be dependent, as in the Dirichlet case, and in Section \ref{sect:pos-gene} it may also depend on $d$. In Section \ref{sect-clas-case} we prove, in a general setting (see Hypothesis \ref{Hyp:base}), that the Euclidean norm of $\vett S_{d,n}$ converges almost surely (a.s.) to the Euclidean norm of the random vector $\vett L$ of the step lengths of the Pearson random walk when the dimension $d$ goes to $+\infty$, without any moment assumption. In Section \ref{sect:pos-gene}, we investigate the same convergence result in a more general framework (see Hypothesis \ref{Hyp:base-var}), where both the number and the lengths of the steps depend on the dimension. Within this setting, we establish both convergence in probability and almost sure convergence under conditions on the mixed moments of the step length random vector and on the growth of the number of steps. It is worth noting that, unlike the classical setting of Section \ref{sect-clas-case}, if the step lengths depend on $d$ then convergence in probability may hold whereas a.s. convergence does not, even for a fixed number of steps; see Example \ref{contro-esempio}, and see also Examples \ref{ex:primo} and \ref{ex:esplosione} for a discussion about the optimality of conditions for convergence in probability in the general setting. Finally, we present some additional examples illustrating applications of these results.

\section{Notation}
Let $(\Omega,\mathcal{F}, \mathbb{P})$ be a complete probability space; all random variables considered from now on are defined on this probability space. Let $E$ be a Euclidean space. We denote by $\mathcal{B}(E)$ the Borel $\sigma$-field of $E$.  Let $X:\Omega\rightarrow E$ be a random variable. We denote by $\mathscr{L}(X)$ the law of $X$ and by $\mathbb{E}\left[X\right]$ the expectation of $X$ with respect to $\mathbb{P}$. Given a probability measure $\mu$ on $(E,\mathcal{B}(E))$, we use the notation $X\sim \mu$ to denote $\mathscr{L}(X)=\mu$. Given another random variable $Y:\Omega\rightarrow E$, we use the notation $X\stackrel{d}{=}Y$ to denote that $\mathscr{L}(X)=\mathscr{L}(Y)$. For $A\subseteq E$, we denote by $\Chi_A$ the indicator function of the set $A$. If $E=\R^d$ with $d\in\N$ and $\mathscr{L}(X)$ is absolutely continuous with respect to the $d$-dimensional Lebesgue measure, then we denote by $f_X$ its density, in this case we say that $X$ is an absolutely continuous random variable with density $f_X$.

We denote by $\left<\cdot,\cdot\right>$ the standard inner product in $\R^d$ and by $S^{d-1}$ the unit sphere in $\R^d$, namely $S^{d-1}:=\{\vett x\in \R^d\ : \ |\vett x|_2=1\}$. We denote by $\Gamma$ the Euler gamma function and by $B$ the beta function. We denote by ${\rm Gauss}(0,1)$, ${\rm Beta}(a,b)$, ${\rm Ber}(p)$ and $\chi^2_{(k)}$ the standard Gaussian, the Beta, the Bernoulli and the chi-squared distributions, respectively. Given two sequences of positive numbers $\{a_d\}$ and $\{b_d\}$, we write $a_d\lesssim b_d$, or $b_d\gtrsim a_d$, if there exists $C>0$, independent of $d$, such that $a_d\leq Cb_d$ for every $d$, and $a_d\approx b_d$ if $a_d\lesssim b_d$ and $b_d\lesssim a_d$.

\section{Properties of the uniform distribution on the sphere}

We recall the definition and we prove some properties of the uniform distribution on the sphere $S^{d-1}$ of $\R^d$; we refer to \cite{Fan} for a detailed study on this topic.

\begin{definition}\label{UnifSd}
Let $\vett V$ be an $\R^d$-valued random variable with $d \geq 2$. We say that $\vett V$ has uniform distribution on the sphere $S^{d-1}$ of $\R^d$ if there exist $Z_1,\ldots, Z_d$ \textnormal{i.i.d.} (independent and identically distributed) ${\rm Gauss}(0,1)$ real random variables such that
	\[
	\vett V\stackrel{d}{=} \left( \frac{Z_1}{\sqrt{\sum_{h=1}^{d} Z_h^2}}, \frac{Z_2}{\sqrt{\sum_{h=1}^{d} Z_h^2}},\dots, \frac{Z_d}{\sqrt{\sum_{h=1}^{d} Z_h^2}}\right).
	\]
	We use the notation $\vett V\sim {\rm Unif}(S^{d-1})$.
\end{definition}

Exploiting this definition, it is quite easy to calculate the marginal distribution of $\vett V$.

\begin{prop}\label{prop:legge}
    Let $\vett V=(V_1,\ldots, V_d)$ be an $\R^d$-valued random variable with $d \geq 2$ such that $\vett V\sim {\rm Unif}(S^{d-1})$. For every $h\in\{1,\ldots,d\}$ the real random variable $V_h$ is absolutely continuous with density
    \[
    f_{V_h}(u):=\frac{1}{B(1/2,(d-1)/2)}\, \left( 1- u^2 \right)^{\frac{d-3}{2}} \Chi_{(-1,1)}(u).
    \]
    In particular for every $h\in\{1,\ldots,d\}$ and $k\in\N\cup \{0\}$ we have
\begin{equation}\label{eq:uniform_marg_moments}
\E\left[ V_h^{2k+1} \right]=0,\qquad\E\left[ V_h^{2k} \right] = \frac{B\left( \frac{1}{2} + k , \frac{d-1}{2} \right)}{B\left( \frac{1}{2}, \frac{d-1}{2} \right)}=\prod_{i=0}^{k-1}\frac{1/2+i}{d/2+i}.
	\end{equation}
For all $h\in\{1,\ldots,d\}$, $v \in (0,1)$ and $d \ge 4$
\begin{equation}
\P\left( |V_h| > v \right) \le 2\sqrt{\frac{d}{\pi}} \exp\left\{ -\frac{\left( d - 3 \right) v^2}{2} \right\},
\label{eq:v_upperbound}
\end{equation}
and for all $v \in (0,1/2)$ and $d\geq 4$
\begin{equation}
\P\left( |V_h| > v \right) \ge v \sqrt{\frac{d}{4 \pi}} \exp\left\{ - \frac{9 (d-3) v^2}{4} \right\}.
\label{eq:v_lowerbound}
\end{equation}

\end{prop}
\begin{proof}
We prove the statement for $h=1$. We remark that $V_1^2 \stackrel{d}{=} Z_1^2/(Z_1^2+\sum_{i=2}^{d} Z_i^2)$, where $Z_1^2\sim\chi^2_{(1)}$ and $\sum_{i=2}^{d} Z_i^2\sim\chi^2_{(d-1)}$ are independent, then $V^2_1\sim {\rm Beta}\left( 1/2, (d-1)/2 \right)$. Hence by the symmetry of the distribution of $V_1$, we obtain the statement.

Now we can prove \eqref{eq:v_upperbound} and \eqref{eq:v_lowerbound}. First of all we note that $|V_1|$ is absolutely continuous with density
\begin{equation}\label{densità}
f_{|V_1|}(u):=\frac{2}{B(1/2,(d-1)/2)}\, \left( 1- u^2 \right)^{\frac{d-3}{2}} \Chi_{(0,1)}(u).
\end{equation}
Further, we recall that for all $u_1 \in (0,1)$, $u_2\in (0,0.78)$ and $d \ge 4$ we have
\begin{equation}
		 1 - u_1 \le e^{- u_1},
	\qquad e^{- 2 u_2} \le 1 - u_2,
	\qquad
	\sqrt{\frac{\pi}{d}} \le B\left( \frac{1}{2}, \frac{d - 1}{2} \right) \le 2 \sqrt{\frac{\pi}{d}}.
	\label{eq:exp_ineq}
\end{equation}

Hence, for every $v\in (0,1/2)$ and $d \ge 4$ by \eqref{eq:exp_ineq}, we have
\begin{align*}
    \mathbb P(|V_1| > v) &= \frac{2}{B(1/2,(d-1)/2)} \int_v^1\left( 1- u^2 \right)^{\frac{d-3}{2}} d u
                \geq \frac{2}{B(1/2,(d-1)/2)}\int_v^{3v/2}\left( 1- u^2 \right)^{\frac{d-3}{2}}du\\
                &\geq \frac{v}{B(1/2,(d-1)/2)}\left(1-\frac{9v^2}{4}\right)^{\frac{d-3}{2}}
                \ge v \sqrt{\frac{d}{4 \pi}} \exp\left\{ - \frac{9 (d-3) v^2}{4} \right\}.
\end{align*}

The upper bound \eqref{eq:v_upperbound} follows immediately from \eqref{densità} and \eqref{eq:exp_ineq}.

\end{proof}

\begin{prop}
	 Let $\vett V=(V_1,\ldots, V_d)$ and $\vett W=(W_1,\ldots, W_d)$  be two  \textnormal{i.i.d.}  ${\rm Unif}(S^{d-1})$ random variables with $d\geq 2$. Then
     \begin{equation}\label{eq:d-scal}
     \left<\vett V,\vett W\right>\stackrel{d}{=} V_1
    \end{equation}
\end{prop}
\begin{proof}
Since  ${\rm Unif}(S^{d-1})$ is rotation invariant, namely $R\vett V \stackrel{d}{=} \vett V$ for every element $R$ of the orthogonal group of $\R^d$, then for every $\vett v\in\R^d$ such that $|\vett v|_2=1$  we have
\begin{equation}\label{rot-inv}
\left<\vett V, \vett v\right> \stackrel{d}{=} \left<\vett V,\vett e_1\right>
\end{equation}
where $\vett e_1$ is the first element of the canonical basis of $\R^d$. Since $\vett V$ and $\vett W$ are independent, by \eqref{rot-inv} and standard properties of conditional expectation, for every $A\in\mathcal{B}(\R)$ we have
\begin{align*}
\mathbb P\left(\left<\vett V,\vett W\right>\in A\right)&=\int_{\R^d}\mathbb P\left(\left<\vett V,\vett W\right>\in A|\vett W=\vett v\right)\mathscr{L}(\vett W)(d\vett v)\\
&=\int_{\R^d}\mathbb P\left(\left<\vett V,\vett v\right>\in A\right)\mathscr{L}(\vett W)(d\vett v)\\
&=\int_{\R^d}\mathbb P\left(\left<\vett V,\vett e_1\right>\in A\right)\mathscr{L}(\vett W)(d\vett v)\\
&=\mathbb P\left(\left<\vett V,\vett e_1\right>\in A\right)\\
&=\mathbb P\left(V_1\in A\right).
\end{align*}
\end{proof}

\section{Asymptotic results for classical Pearson random walks}\label{sect-clas-case}
In this section, we study the asymptotic behaviour of the Euclidean norm of the classical Pearson random walk in $\R^d$ when $d$ goes to $+\infty$. First, we state the assumptions under which we will work.

\begin{hyp}\label{Hyp:base}
    Let $n\in\N$.
    \begin{enumerate}
        \item Let $\vett U^{(d)}_1,\ldots, \vett U^{(d)}_n$ be  \textnormal{i.i.d.}  ${\rm Unif}(S^{d-1})$ $\R^d$-valued random variables with $d\geq 2$, see Definition \ref{UnifSd}.
        \item\label{hyp:base2} Let $\vett L = \left(L_1,\ldots, L_n\right)$ be a random vector of non-negative random variables which does not depend on $d$.
        \item\label{Hyp:base3} $\vett U^{(d)}_1,\ldots, \vett U^{(d)}_n$ and $\vett L$ are independent.
    \end{enumerate}
\end{hyp}

We denote by $\vett S_{d,n}$ the $\R^d$-valued random variable defined by
\[
\vett S_{d,n}:=\sum_{k=1}^nL_k\vett U^{(d)}_k.
\]
We are going to study the asymptotic behaviour of $|\vett S_{d,n}|_2$ when $d$ goes to $+\infty$. Since $|\vett U^{(d)}_k|_2=1$ for every $k\in\{1,\ldots,n\}$, we deduce
\begin{equation}\label{formula-Sdn}
|\vett S_{d,n}|_2^2=\left<\vett S_{d,n},\vett S_{d,n}\right>=\sum_{k=1}^nL_k^2+\sum_{i\neq j\in\{1,\ldots, n\}}L_iL_j\left<\vett U^{(d)}_i,\vett U^{(d)}_j\right>.
\end{equation}
Since we are going to study almost sure convergence, we introduce a common probability space where all the random variables of the framework are defined:
let $\left\{Z_{i,h}\, :\, i\in \{1,\dots, n\}, h \in\N \right\}$ be a family of i.i.d. ${\rm Gauss}(0,1)$ real valued random variables, independent of $\vett L$, defined on the same complete probability space $(\Omega,\mathcal{F},\mathbb P)$. In view of Definition \ref{UnifSd}, for every $d\geq 2$ and $i\in\{1,\ldots, n\}$ setting
\begin{equation}\label{eq:costruzioneSd}
\vett U^{(d)}_i := \left( \frac{Z_{i,1}}{\sqrt{\sum_{h=1}^{d} Z_{i,h}^2}}, \frac{Z_{i,2}}{\sqrt{\sum_{h=1}^{d} Z_{i,h}^2}},\dots, \frac{Z_{i,d}}{\sqrt{\sum_{h=1}^{d} Z_{i,h}^2}}\right),
\end{equation}
we have a family of random vectors with the above properties and $\{|\vett S_{d,n}|_2\}_{d\geq 2}$, defined as above, is a sequence of real valued random variables defined on $(\Omega,\mathcal{F},\mathbb P)$.

\begin{remark}\label{rmk:gen}$ $
\begin{enumerate}
    \item We emphasize that Definition \ref{UnifSd} cannot be generalized to define a uniform distribution on the sphere of an infinite-dimensional space. This fact is closely related to the behaviour of Gaussian measures in infinite-dimensional settings (see \cite{boogie}) and to the absence of non-degenerate rotation-invariant measures on an infinite-dimensional separable Hilbert space $H$ (see \cite[Theorem 1.5]{Kuku}). For this reason, we cannot study an infinite-dimensional version of $\vett S_{d,n}$, but we can only study the behaviour of $|\vett S_{d,n}|_2$ when $d$ goes to $+\infty$.

    \item By construction the families
    \[
    \left\{\vett U^{(d)}_i\, :\, i\in\{1,\ldots,n\}\right\}_{d\geq 2 }
    \]
    are not independent. However, it is possible to modify the framework to make the families independent, indeed it is sufficient to take a family $\left\{Z_{i,h,k}\, :\, i\in \{1,\dots, n\},\, h\in\N,\, k\in\N \right\}$ \textnormal{i.i.d.} ${\rm Gauss}(0,1)$  defined on the same probability space $(\Omega,\mathcal{F},\mathbb P)$ and set
    \begin{equation*}
\vett U^{(d)}_i=\left( \frac{Z_{i,1,d}}{\sqrt{\sum_{h=1}^{d} Z_{i,h,d}^2}}, \frac{Z_{i,2,d}}{\sqrt{\sum_{h=1}^{d} Z_{i,h,d}^2}},\dots, \frac{Z_{i,d,d}}{\sqrt{\sum_{h=1}^{d} Z_{i,h,d}^2}}\right)
\end{equation*}
for every $d\geq 2$ and $i\in\{1,\ldots, n\}$.

    We note that, even in this case, in general $\left\{|\vett S_{d,n}|_2^2\right\}_{d \geq 2}$ are not independent, since  $\vett L$ is a fixed random vector independent of $d$. In the next section, we are going to consider also the case where $\vett L$ may depend on $d$ and $\left\{|\vett S_{d,n}|_2^2\right\}_{d \geq 2}$ could be a sequence of independent random variables, see Example \ref{contro-esempio}.
\end{enumerate}

\end{remark}

\begin{theorem}\label{thm: as}
    Assume Hypothesis \ref{Hyp:base} holds. Then
    \begin{equation}\label{conv as}
        \lim_{d\rightarrow+\infty}|\vett S_{d,n}|_2^2=|\vett L|_2^2,\qquad \mathbb P\mbox{-a.s}.
    \end{equation}
\end{theorem}
\begin{proof}
First of all, we prove that, for every $i,j\in\{1,\ldots, n\}$ such that $i\neq j$ we have
\begin{equation}\label{conv as-U}
        \lim_{d\rightarrow+\infty}\left|\left< \vett U_i^{(d)} , \vett U_j^{(d)} \right> \right|= 0,\qquad \mathbb P\mbox{-a.s}.
    \end{equation}
By \eqref{eq:v_upperbound} and \eqref{eq:d-scal}, for every $i,j\in\{1,\ldots, n\}$ such that $i\neq j$ and for every $\varepsilon\in (0,1)$ we deduce
    \begin{align*}
        \sum_{d=2}^{+\infty}\P\left( \left|\left< \vett U_i^{(d)} , \vett U_j^{(d)} \right> \right|> \varepsilon \right)\leq 2+\sum_{d=4}^{+\infty}2\sqrt{\frac{d}{ \pi}} \exp\left\{ - \frac{ (d-3) \varepsilon^2}{2} \right\},
    \end{align*}
then, by the Borel--Cantelli lemma applied with $\varepsilon=1/h$, $h\in\N$, \eqref{conv as-U} holds, in particular there exists a finite family $\{\Omega_{i,j}\}_{i\neq j\in\{1,\ldots,n\}}\subseteq \mathcal{F}$ such that $\mathbb P(\Omega_{i,j})=1$ and for every $\omega\in\Omega_{i,j}$ we have
    \begin{align}\label{eq:limit}
        \lim_{d\rightarrow +\infty}\left|\left< \vett U^{(d)}_i(\omega) , \vett U^{(d)}_j(\omega) \right> \right|=0.
    \end{align}
    Letting $\Omega_0 = \cap_{i\neq j\in\{1,\ldots, n\}}\Omega_{i,j}$, we have $\P(\Omega_0)=1$ and  for every $\omega\in\Omega_{0}$ and $i,j\in\{1,\ldots, n\}$ such that $i\neq j$ formula \eqref{eq:limit} holds. Finally, noting that $\vett L$ does not depend on $d$ (see Hypothesis \ref{Hyp:base}), by \eqref{formula-Sdn} and \eqref{eq:limit} we get
    \begin{align*}
        \limsup_{d\rightarrow +\infty} \left| |\vett S_{d,n}(\omega)|_2^2-|\vett L(\omega)|^2_2 \right|\leq \sum_{i\neq j\in\{1,\ldots, n\}}L_i(\omega)L_j(\omega) \lim_{d\rightarrow +\infty} \left|\left<\vett U^{(d)}_i(\omega),\vett U^{(d)}_j(\omega)\right>\right|=0,
    \end{align*}
for every $\omega\in\Omega_{0}$ and therefore \eqref{conv as} holds.
\end{proof}

\begin{remark}\label{rmk} $ $
 If $\vett L$ is bounded, by \eqref{formula-Sdn}, \eqref{eq:d-scal} and \eqref{eq:v_upperbound} it is easy to prove that $\P\left(\left||\vett S_{d,n}|_2^2-|\vett L|_2^2\right|>\varepsilon\right)$ decays exponentially in $d$. However, in general, such exponential decay fails. Indeed, although the convergence rate of $\left<\vett U^{(d)}_i,\vett U^{(d)}_j\right>$ to $0$ is exponential, when multiplied by a random variable $L$, the tail behaviour of $L$ can significantly modify the convergence rate of $L_iL_j\left<\vett U^{(d)}_i,\vett U^{(d)}_j\right>$; see Example \ref{ex:rate}.
\end{remark}

If we have additional conditions on the moments of $\vett L$, then we can prove further convergence results.

\begin{theorem}\label{thm: momenti}
    Assume Hypothesis \ref{Hyp:base} holds and there exists $m\in\N$ such that $\E \left[L_i^{m}L_j^{m}\right]<+\infty$ for every $i\neq j\in \{1,\ldots, n\}$. Then
    \begin{equation}\label{eq:E-conv}
    \lim_{d\rightarrow +\infty}\mathbb E\left[\left||\vett S_{d,n}|_2^2-|\mathbf L|_2^2\right|^m\right]=0.
    \end{equation}
\end{theorem}

\begin{proof}
By \eqref{formula-Sdn} and the Jensen inequality, we have
\begin{align*}
    \mathbb E\left[\left||\vett S_{d,n}|_2^2-|\mathbf L|_2^2\right|^m\right]&=\mathbb E\left[\left| \sum_{i\neq j\in\{1,\ldots, n\}}L_iL_j\left<\vett U_i^{(d)},\vett U_j^{(d)}\right>\right|^m\right]\\
    &\leq [n(n-1)]^{m-1}\sum_{i\neq j\in\{1,\ldots, n\}}\mathbb E\left[|L_iL_j\left<\vett U_i^{(d)},\vett U_j^{(d)}\right>|^m\right],
\end{align*}
by Hypothesis \ref{Hyp:base}\eqref{Hyp:base3} and the H\"older inequality we obtain
\begin{align*}
    \mathbb E\left[\left||\vett S_{d,n}|_2^2-|\mathbf L|_2^2\right|^m\right]&\leq [n(n-1)]^{m-1}\sum_{i\neq j\in\{1,\ldots, n\}}\mathbb E\left[L_i^mL_j^m\right]\sqrt{\mathbb E\left[|\left<\vett U_i^{(d)},\vett U_j^{(d)}\right>|^{2m}\right]}.
\end{align*}
Finally, by \eqref{eq:uniform_marg_moments} and \eqref{eq:d-scal}, there exists a constant $c_{m}>0$ such that
\[
\mathbb E\left[\left||\vett S_{d,n}|_2^2-|\mathbf L|_2^2\right|^m\right]\leq c_{m}\frac{[n(n-1)]^{m-1}}{(d-1)^{m/2}}\sum_{i\neq j\in\{1,\ldots, n\}}\mathbb E\left[L_i^mL_j^m\right],
    \]
    hence \eqref{eq:E-conv} holds.
\end{proof}

\section{Asymptotic results in the general case}\label{sect:pos-gene}
In this section, we propose a generalization of the framework of the present paper to the case where the number and length of steps in the Pearson random walk also depend on the dimension $d \in \mathbb{N}$ of the sphere from which the directions are sampled.

\begin{hyp}\label{Hyp:base-var}
 Let $\{n_d\}_{d\geq 2}\subseteq \N$ be a non-decreasing sequence. For every $d\geq 2$
    \begin{enumerate}
        \item let $\vett U^{(d)}_1,\ldots, \vett U^{(d)}_{n_d}$ be \textnormal{i.i.d.} ${\rm Unif}(S^{d-1})$ random variables;
	\item let $\vett L^{(d)} = \left( L^{(d)}_1,\ldots, L^{(d)}_{n_d} \right)$ be a random vector of non-negative random variables;
        \item $\vett U^{(d)}_1,\ldots, \vett U^{(d)}_{n_d}$ and $\vett L^{(d)}$ are independent.
    \end{enumerate}
\end{hyp}

As before, we denote by $\vett S_{d,n_d}$ the $\R^d$-valued random variable defined by
\[
\vett S_{d,n_d}:=\sum_{k=1}^{n_d}L^{(d)}_k\vett U^{(d)}_k.
\]
Moreover, we note that
\begin{equation}\label{formula-Sdn-Var}
|\vett S_{d,n_d}|_2^2=\left<\vett S_{d,n_d}, \vett S_{d,n_d}\right>=\sum_{k=1}^{n_d} \left(L^{(d)}_k\right)^2+\sum_{i\neq j\in\{1,\ldots, n_d\}}L^{(d)}_iL^{(d)}_j\left<\vett U_i^{(d)},\vett U_j^{(d)}\right>.
\end{equation}

As in the previous section, we can introduce a common probability space where all the random variables of the framework are defined using a construction analogous to \eqref{eq:costruzioneSd}.

\begin{remark}
    Obviously, the framework defined by Hypothesis \ref{Hyp:base-var} covers the case in which we consider a sequence $\{L_n\}_{n\in\N}$ of non-negative random variables and set \[ \vett L^{(d)} = \left( L_1,\ldots, L_{n_d} \right). \] However, we choose to work in this more general setting, since in many significant examples the distribution of the components of $\vett L^{(d)}$ necessarily depends on $d$; see Examples \ref{contro-esempio} and \ref{Poiss}.
\end{remark}

In this general setting we cannot perform a proof similar to the one of Theorem \ref{thm: as}. In particular we will see that in this case there is a gap between convergence in probability and $\P$-a.s. convergence, see  Example \ref{contro-esempio}. The following lemma is the key point to prove the convergence results in this general setting.
\begin{lemma}\label{prop:momenti}
Let $m\in\N$ and let $\vett U_1,\ldots,\vett U_m$ be \textnormal{i.i.d.} ${\rm Unif}(S^{d-1})$ random variables with $d\geq 2$. Then for every $r\geq1$ there exists $c_r>0$, depending only on $r$, such that for every $\ell_1,\ldots,\ell_{m}\in[0,+\infty)$
\begin{equation}\label{eq:momenti}
\mathbb E\left[\left|\sum_{i\neq j\in\{1,\ldots,m\}}\ell_i\ell_j\left<\vett U_i,\vett U_j\right>\right|^r\right]\leq c_r\left(\frac{1}{d}\sum_{1\leq i<j\leq m}\ell_i^2\ell_j^2\right)^{r/2}\leq  c_rm^{\max\{r-2,0\}}\frac{1}{d^{r/2}}\sum_{1\leq i<j\leq m}\ell_i^r\ell_j^r .
\end{equation}
\end{lemma}
\begin{proof}
By the H\"older inequality, for every $r\geq1$ and $p\in\N$ such that $2p\geq r$ we have
\[
\mathbb E\left[\left|\sum_{i\neq j\in\{1,\ldots,m\}}\ell_i\ell_j\left<\vett U_i,\vett U_j\right>\right|^r\right] \leq\mathbb E\left[\left(\sum_{i\neq j\in\{1,\ldots,m\}}\ell_i\ell_j\left<\vett U_i,\vett U_j\right>\right)^{2p}\right]^{\frac{r}{2p}},
\]
hence it is enough to prove the first inequality in \eqref{eq:momenti} for $r=2p$ with $p\in\N$, and then it holds for every $r\geq1$ with $c_r=c_{2p}^{r/(2p)}$, where $p=\lceil r/2\rceil$. We set
\[
\mathcal{J}:=\left\{(i,j)\in \{1,\ldots,m\}^2\ :\ i\neq j\right\}.
\]
By standard calculations, we have
\begin{align*}
&\mathbb E\left[\left(\sum_{i\neq j\in\{1,\ldots,m\}}\ell_i\ell_j\left<\vett U_i,\vett U_j\right>\right)^{2p}\right]
=\mathbb E\left[\prod_{s=1}^{2p}\left(\sum_{i_s\neq j_s\in\{1,\ldots,m\}}\ell_{i_s}\ell_{j_s}\left<\vett U_{i_s},\vett U_{j_s}\right>\right)\right]\\
%&\qquad=\sum_{(i_1,j_1,\ldots,i_{2p},j_{2p})\in \mathcal{J}}\ell_{i_1}\ell_{j_1}\cdots\ell_{i_{2p}}\ell_{j_{2p}}\,\mathbb E\left[\left<\vett U_{i_1},\vett U_{j_1}\right>\cdots\left<\vett U_{i_{2p}},\vett U_{j_{2p}}\right>\right]\\
&\qquad=  \sum_{(i_1,j_1)\in \mathcal{J}}\dots\sum_{(i_{2p},j_{2p} )\in \mathcal{J}} \prod_{s = 1}^{2p} \ell_{i_s}\ell_{j_s}\,\mathbb E\left[ \prod_{w = 1}^{2p} \left<\vett U_{i_w},\vett U_{j_w}\right>\right].
\end{align*}
In general, by the H\"older inequality, \eqref{eq:uniform_marg_moments} and \eqref{eq:d-scal}, we have
\[
\left|\mathbb E\left[\prod_{w=1}^{2p}\left<\vett U_{i_w},\vett U_{j_w}\right>\right]\right|\leq\prod_{w=1}^{2p}\mathbb E\left[\left<\vett U_{i_w},\vett U_{j_w}\right>^{2p}\right]^{\frac{1}{2p}}
=\mathbb E\left[V_1^{2p}\right]\leq\frac{(2p-1)!!}{d^p}
=\frac{(2p)!}{p!2^pd^p}.
\]
Moreover, by independence and symmetry 
 $\left(\vett U_1,\ldots,\vett U_v,\ldots, \vett U_{m}\right)\stackrel{d}{=}\left(\vett U_1,\ldots,-\vett U_v,\ldots, \vett U_{m}\right)$
for all $v \in \left\{ 1,\dots, m \right\}$, hence for all $((i_1, j_1), \ldots, (i_{2p}, j_{2p}) )\in \mathcal{J}^{2p}$ such that there exists $h \in \left\{ 1,\dots, 2p \right\}$ for which $i_h$ or $j_h$ have odd multiplicity in $(i_1,j_1,\ldots,i_{2p},j_{2p})$ we have
\[
\mathbb E\left[ \prod_{w = 1}^{2p} \left<\vett U_{i_w},\vett U_{j_w}\right> \right] = 0.
\]
Hence only the $2p$-tuples in which every index has even multiplicity remain (rewrite the entries in non-decreasing order as $(c_1,c_1,\ldots,c_{2p},c_{2p})$ and pair $c_t$ with $c_{t+p}$, which are distinct), and so, taking into account that at most $(4p)!/2^{2p}$ elements of $\mathcal J^{2p}$ have the same entries, since these form $2p$ pairs of equal values, we get
\begin{align*}
	\mathbb E\left[\left(\sum_{i\neq j\in\{1,\ldots,m\}}\ell_i\ell_j\left<\vett U_i,\vett U_j\right>\right)^{2p}\right] & \leq \frac{(2p)!(4p)!}{2^{3p}p!d^p} \sum_{(i_1,j_1)\in \mathcal{J}}\dots\sum_{(i_{p},j_{p} ) \in\mathcal{J}} \prod_{s = 1}^{p} \ell_{i_s}^2 \ell_{j_s}^2\\ 
&=\frac{(2p)!(4p)!}{2^{3p}p!d^p}  \left(\sum_{i \neq j\in\{1,\ldots,m\}}\ell^2_{i}\ell^2_{j}\right)^p\\
&=\frac{(2p)!(4p)!}{2^{2p}p!d^p}  \left(\sum_{1\leq i <j\leq m}\ell^2_{i}\ell^2_{j}\right)^p
\end{align*}
The first inequality in \eqref{eq:momenti} holds with $r=2p$ and $c_r=(2p)!(4p)!/[2^{2p}p!]$ and the second one follows from the standard inequality $\big(\sum_{h=1}^{N}x_h\big)^{q}\leq N^{\max\{q-1,0\}}\sum_{h=1}^Nx_h^{q}$ for every $N\in\N$, $q>0$ and $x_1,\ldots,x_N\geq0$, applied with $q=r/2$ and $N=m(m-1)/2\leq m^2/2$.
\end{proof}

In the next theorem we study the convergence of $|\vett S_{d,n_d}|_2^2-|\vett L^{(d)}|_2^2$ to $0$, in probability and $\P$-a.s.

\begin{theorem}\label{thm:C-Prob}\label{thm:C-AS}
Assume Hypothesis \ref{Hyp:base-var} holds and there exist $k\in\N$ and $\psi:\N\rightarrow\R^+$ such that
\begin{equation}\label{cond-P1}
\mathbb{E}\left[\left(L^{(d)}_i L^{(d)}_j\right)^k\right]\leq \psi(d),\qquad d\in\N,\, i\neq j\in\{1,\ldots,n_d\}.
\end{equation}
\begin{enumerate}
\item\label{thm:C-1} If the following condition holds
\begin{equation}\label{cond-P2}
\lim_{d\rightarrow+\infty}d^{-k/2}\psi(d)n_d^{\max\{k,2\}}=0,
\end{equation}
then $|\vett S_{d,n_d}|_2^2-|\vett L^{(d)}|_2^2$ converges to $0$ in probability as $d\to+\infty$, namely for every $\varepsilon> 0$ we have
    \begin{equation}\label{eq:P-lim}
            \lim_{d\rightarrow +\infty}\P\left( \left| |\vett S_{d,n_d}|_2^2-|\vett L^{(d)} |^2_2 \right|> \varepsilon \right)=0.
    \end{equation}

\item\label{thm:C-2} If there exists $\gamma>0$ such that
\begin{equation}\label{condAS-P2}
\lim_{d\rightarrow+\infty}d^{1+\gamma-k/2}\psi(d)n_d^{\max\{k,2\}}=0,
\end{equation}
then $|\vett S_{d,n_d}|_2^2-|\vett L^{(d)}|_2^2$ converges to $0$ $\P$-a.s. as $d\to+\infty$, namely
\begin{equation}\label{eq:AS-lim}
	   \P\left(\lim_{d\rightarrow+\infty}\left(|\vett S_{d,n_d}|_2^2-|\vett L^{(d)}|_2^2\right)=0\right)=1.
    \end{equation}
\end{enumerate}
\end{theorem}

\begin{proof}
Fix $\varepsilon>0$ and $d\geq 2$. We use the notation $\ell=(\ell_1,\ldots,\ell_{n_d})$. By \eqref{formula-Sdn-Var} and Hypothesis \ref{Hyp:base-var}(3) we obtain
\begin{align*}
\P\left( \left| |\vett S_{d,n_d}|_2^2-|\vett L^{(d)} |^2_2 \right|> \varepsilon \right)
&=\int_{[0,+\infty)^{n_d}}\P\left(\left|\sum_{i\neq j\in\{1,\ldots,n_d\}}L^{(d)}_iL^{(d)}_j\left<\vett U^{(d)}_i,\vett U^{(d)}_j\right>\right|>\varepsilon\ |\ \vett L^{(d)}=\vett \ell\right)\mathscr{L}(\vett L^{(d)})(d\ell)\\
&=\int_{[0,+\infty)^{n_d}}\P\left(\left|\sum_{i\neq j\in\{1,\ldots,n_d\}}\ell_i\ell_j\left<\vett U^{(d)}_i,\vett U^{(d)}_j\right>\right|>\varepsilon\right)\mathscr{L}(\vett L^{(d)})(d\ell),
\end{align*}
hence by the Markov inequality, \eqref{eq:momenti} with $r=k$ and $m=n_d$ and \eqref{cond-P1} we get
\begin{equation}\label{eq:stima-unica}
\P\left( \left| |\vett S_{d,n_d}|_2^2-|\vett L^{(d)} |^2_2 \right|> \varepsilon \right)\leq\frac{c_k}{\varepsilon^k}\,d^{-k/2}\psi(d)n_d^{\max\{k,2\}}.
\end{equation}
\eqref{eq:P-lim} follows directly from \eqref{cond-P2} and \eqref{eq:stima-unica}. Instead, \eqref{eq:AS-lim} follows from \eqref{condAS-P2}, \eqref{eq:stima-unica} and the Borel--Cantelli lemma.
\end{proof}

\begin{remark}
    We note that the assumptions of Theorem \ref{thm:C-Prob} do not guarantee that $\lim_{d\to+\infty}|\vett L^{(d)}|_2^2$ is a real random variable; however, the convergence result still holds even in the case where $|\vett L^{(d)}|_2^2$ diverges, see Example \ref{ex:esplosione}.
\end{remark}

\begin{remark}
   In Example \ref{contro-esempio}, we present a case where \eqref{condAS-P2} is not satisfied and $\P$-a.s. convergence does not hold. See also Example \ref{ex:esplosione} for a consideration about the optimality of condition \eqref{cond-P2} for convergence in probability.
\end{remark}

\section{Examples}

In this section, we present some significant examples.

\begin{example}\label{ex:primo}
This example is an interesting application of the first inequality in \eqref{eq:momenti}. Let $\{n_d\}_{d\in \N}\subseteq \N$ be a non-decreasing sequence and, for every $d\geq 2$, let $\vett L^{(d)}= \left(1,1/2,\ldots, 1/n_d \right)$. In this case $|\vett L^{(d)}|_2=\sqrt{\sum_{h=1}^{n_d}h^{-2}}$ for every $d\geq 2$. Since $\vett L^{(d)}$ is deterministic, by the Markov inequality and the first inequality in \eqref{eq:momenti} with $r=2$, for every $\varepsilon>0$ we have
    \[
    \mathbb P\left(\left||\vett S_{d,n_d}|^2_2-\sum_{h=1}^{n_d}h^{-2}\right|>\varepsilon\right) \leq \frac{c_2}{\varepsilon^2d}\sum_{1\leq i<j\leq n_d}\frac{1}{i^2j^2}
    < \frac{c_2}{\varepsilon^2d} \left( \sum_{i = 1}^{n_d} \frac{1}{i^2} \right)^2
    \]
    and in particular
    \begin{equation}
    \lim_{d\rightarrow +\infty}\mathbb P\left(\left||\vett S_{d,n_d}|^2_2-\sum_{h=1}^{n_d}h^{-2}\right|>\varepsilon\right)=0.
    \end{equation}
    Note that Theorem \ref{thm:C-Prob} can be applied only if $n_d=o(d^{1/2})$, since \eqref{cond-P1} bounds each product $L^{(d)}_iL^{(d)}_j$ uniformly and does not take into account the decay of the step lengths, which here makes $\sum_{1\leq i<j\leq n_d}i^{-2}j^{-2}$ bounded uniformly in $n_d$.
\end{example}

\begin{example}\label{ex:esplosione}
	Let $\{n_d\}_{d\in \N}\subseteq \N$ be a non-decreasing sequence and, for every $d\geq 2$, let $\vett L^{(d)}= \left( n_d^{-\eta},\ldots, n_d^{-\eta} \right)\in\R^{n_d}$ for some $\eta\geq 0$. In this case $|\vett L^{(d)}|_2=n_d^{1/2-\eta}$ for every $d\geq 2$. In the particular case where $\eta=1/2$ and $n_d=n\in\N$ is constant for every $d\geq 2$, then Theorem \ref{thm: as} is applicable and therefore
 \begin{equation*}
    \lim_{d\rightarrow +\infty}|\vett S_{d,n}|_2=1, \qquad \mathbb P{\rm -a.s.}
    \end{equation*}
    In the general case, we need  $n_d\approx d^{\beta}$ with $\beta>0$ and $(1 - 2 \eta)\beta<1/2$, then setting $\psi(d)=n_d^{-2k\eta}$ there exist $\gamma>0$ and $k \in \N$ such that the assumption \eqref{condAS-P2} of Theorem \ref{thm:C-Prob} holds and therefore we have
    \begin{equation}\label{divergenza}
    \P\left(\lim_{d\rightarrow +\infty}\left(|\vett S_{d,n_d}|_2^{2}-n_d^{1-2\eta}\right)=0\right)=1.
    \end{equation}

    In particular, if $0\leq\eta < 1/2$ then  $|\vett S_{d,n_d}|_2$ diverges $\P$-a.s., and \eqref{divergenza} gives an estimate of the rate of divergence. If $\eta\geq 1/2$ then \eqref{divergenza} hold for every sequence $\{n_d\}_{d\in\N}$: for $\eta=1/2$, so this extends the particular case above to non-constant $n_d$. 

    For $0\leq\eta<1/2$ the condition $(1-2\eta)\beta<1/2$ cannot be improved, even for the convergence in probability, indeed taking $\beta = \left( 2 (1-2\eta) \right)^{-1}$, by \cite[Theorem 3.3(b)]{KabluchkoMarynych2024} applied to the walk $n_d^{\eta}\vett S_{d,n_d}$, $\sqrt{d/2}\,n_d^{2\eta-1}\big(|\vett S_{d,n_d}|_2^2-n_d^{1-2\eta}\big)$ converges in distribution to ${\rm Gauss}(0,1)$ as $n_d\to+\infty$.
\end{example}

\begin{example}\label{ex:rw_gauss_steps}
	In this example, we examine the case of shrinking Pearson random walk studied in \cite{Serino:2010p20599}. In this random walk the lengths of the steps are deterministic and equal to $L_i:=\lambda^{i-1},\,i\ge 1$ and $\lambda\in (0,1)$. From Theorem \ref{thm: as} we deduce that the limiting distribution after $n$ steps is the Dirac distribution at $\sqrt{\frac{\left( 1-\lambda^{2n} \right)}{1-\lambda^2}}$.
    It is interesting that in high dimensions, the conclusions of \cite{Serino:2010p20599} on the modes of low--dimensional random walks do not hold.
    \end{example}

\begin{example}\label{ex:rate}
Here we present a significant example in which the rate of convergence of $|\vett S_{d,2}|^2_2-|\vett L|_2^2$ to $0$ is not exponential. Let $n=2$,  $L_1=1$ $\mathbb P$-a.s. and let $L_2$ be a discrete random variable with distribution $\rho_a$ given by
\[
\rho_{a}(m)=\frac{M}{m^{a+1}}, \qquad m\in\N,
\]
where $a\in (0,2)$ and $M^{-1} = \sum_{m=1}^{+\infty}m^{-(1+a)}$. In this case, one can explicitly compute a lower bound for the rate of convergence of $|\vett S_{d,2}|^2_2-|\vett L|_2^2$ to $0$.

For every $\varepsilon\in (0,1)$ and $d\geq 4$, by \eqref{formula-Sdn} and \eqref{eq:v_lowerbound} we obtain
\begin{align*}
\mathbb{P}\left(\left\vert|\vett S_{d,2}|^2_2-|\vett L|_2^2\right\vert>\varepsilon\right)&\geq\mathbb{P}\left(2L_2\left|\left<\vett U^{(d)}_1,\vett U^{(d)}_2\right>\right|>\varepsilon\right)\\
&=\sum_{m=1}^{+\infty}\mathbb{P}\left(\left|\left<\vett U^{(d)}_1,\vett U^{(d)}_2\right>\right| > \frac{\varepsilon}{2m}\right)\mathbb{P}\left(L_2=m\right)\\
&\geq  \sqrt{\frac{d}{16 \pi}} M \varepsilon \sum_{m=1}^{+\infty}\exp\left\lbrace-\frac{9 d \varepsilon^2}{16 m^2}\right\rbrace\frac{1}{m^{2+a}}\\
&\ge \sqrt{\frac{d}{16 \pi}} M \varepsilon \sum_{m=\lceil \sqrt{d} \rceil}^{+\infty}\exp\left\lbrace - \frac{9 d \varepsilon^2}{16 m^2} \right\rbrace\frac{1}{m^{2+a}}\\
&\geq \sqrt{\frac{d}{16 \pi}} M \varepsilon \exp\left\lbrace- 9\varepsilon^2/16\right\rbrace\sum_{m=0}^{+\infty}\frac{1}{(m+\lceil \sqrt{d} \rceil)^{2+a}}.
\end{align*}
Finally, since $\sum_{m=0}^{+\infty}(m+\lceil\sqrt d\rceil)^{-2-a}\geq\int_{0}^{+\infty}(y+\lceil\sqrt{d}\rceil)^{-2-a}dy=(1+a)^{-1}\lceil\sqrt d\rceil^{-(1+a)}$ and $\lceil\sqrt d\rceil\leq 2\sqrt d$, we can conclude that
\begin{equation}\label{eq:stima-LB}
\mathbb{P}\left(\left||\vett S_{d,2}|^2_2-|\vett L|_2^2\right|>\varepsilon\right)\gtrsim \frac{\varepsilon}{(1+a)d^{a/2}}.
\end{equation}
\end{example}

\begin{example}\label{contro-esempio}
    Let $\{X_d\}_{d\in\N}$ be a sequence of independent ${\rm Ber}(1/\sqrt{d})$ random variables. For every $d\geq 2$, let $n_d=2$,  $L_1^{(d)}=L_1=1$ $\mathbb P$-a.s. and $L^{(d)}_2=d^{\beta}X_d$, with $\beta\in\R$. For every $k\in\N$, we have
    \[
    \mathbb E\left[L^k_1\right]=1,\qquad \mathbb E\left[\left(L^{(d)}_2\right)^k\right]=d^{k\beta-1/2}.
    \]
    Hence, for every $\beta<1$ condition \eqref{cond-P2} hold with $k=1$. Instead, we need $\beta<1/2$ to guarantee that there exists $k\in\N$ such that \eqref{condAS-P2} holds and therefore the second statement of Theorem \ref{thm:C-AS} can be applied. Indeed, we are going to prove that, if $\beta=1/2$, then the sequence
    \[
    Y_d:=|\vett S_{d,2}|^2_2-|\vett L^{(d)}|_2^2=2d^{\beta}X_d\left<\vett U^{(d)}_1,\vett U^{(d)}_2\right>
    \]
   does not converge to $0$ $\mathbb P$-a.s., this fact shows that, in this case, the condition \eqref{condAS-P2} is sharp for the $\mathbb P$-a.s. convergence. For every $\varepsilon>0$ and $d\geq4$ such that $d>\varepsilon^2$, by \eqref{eq:v_lowerbound} we have
    \begin{align*}
        \mathbb{P}\left(|Y_d|>\varepsilon\right)&=  \mathbb{P}\left(\sqrt{d}\left|\left<\vett U^{(d)}_1,\vett U^{(d)}_2\right>\right|>\frac{\varepsilon}{2}\, \Big|\, X_d=1\right)\P\left(X_d=1\right)\\
        &\geq \varepsilon\sqrt{\frac{1}{16\pi d}}\exp\left\{-\frac{9(d-3)\varepsilon^2}{16d}\right\},
    \end{align*}
    hence
    \begin{equation}\label{borel-cantelli-2}
    \sum_{d=4}^{+\infty}\P\left(|Y_d|>\varepsilon\right)=+\infty.
    \end{equation}
    In this example $\{\vett L^{(d)}\}_{d\in\N}$ is a sequence of independent random variables, hence, using the construction of Remark \ref{rmk:gen}(2) with $\{X_d\}_{d\in\N}$ independent of $\{Z_{i,h,k}\}$, also $\{Y_d\}_{d\in\N}$ is a sequence of independent random variables. By \eqref{borel-cantelli-2} and the second Borel--Cantelli lemma we conclude
    \[
    \P\left(\limsup_{d\to+\infty}\{|Y_d|>\varepsilon\}\right)=1,
    \]
    namely $\{Y_d\}_{d\in\N}$ does not converge to $0$ $\P$-a.s.
\end{example}

\begin{example}\label{Poiss}
Let $n\in\N$ and let $L_1,\ldots, L_n$ be \textnormal{i.i.d.} ${\rm Ber}(p)$ random variables, with $p\in [0,1]$. Then $|\vett L|^2_2$ is a Binomial random variable with parameters $n$ and $p$, so Theorem \ref{thm: as} applies and, in particular, also $\lim_{d\rightarrow +\infty}|\vett S_{d,n}|_2^2$ is a Binomial random variable with parameters $n$ and $p$.\\
In the same way, for every $d\geq 2$ such that $d\geq\lambda$, let $\vett L^{(d)}=\left(L^{(d)}_1,\ldots, L^{(d)}_d\right)$ be such that  $L^{(d)}_1,\ldots, L^{(d)}_d$ are \textnormal{i.i.d.} ${\rm Ber}(\lambda/d)$ random variables for some $\lambda>0$. Then \eqref{cond-P2} holds with $k=1$, the first statement of Theorem \ref{thm:C-Prob} applies and, since $|\vett L^{(d)}|_2^2$ converges in distribution to a Poisson random variable with parameter $\lambda$, by the Slutsky lemma also $|\vett S_{d,d}|_2^2$ converges in distribution to a Poisson random variable with parameter $\lambda$.
\end{example}

\noindent\textbf{Declaration of generative AI use.}
The authors used a generative AI assistant to check the English
grammar and to detect misprints. The authors reviewed and
edited the output and take full responsibility for the content of
the article.

% Bibliography as a list of \bibitem (replaces \bibliographystyle{cas-model2-names} and \bibliography{biblio}).
% Author-year format compatible with natbib [authoryear,longnamesfirst]: the optional argument is Short(Year)Long.

\end{document}